\documentclass[a4paper, 11pt]{amsart}
\usepackage{amsmath,amsfonts,amssymb,amsthm}
\usepackage{enumitem}
\usepackage{hyperref}

\newtheorem{thm}{Theorem}[section]
\newtheorem{prop}[thm]{Proposition}
\newtheorem{lem}[thm]{Lemma}
\newtheorem{fact}[thm]{Fact}
\newtheorem{ex}[thm]{Example}

\theoremstyle{definition}
\newtheorem{defn}[thm]{Definition}
\newtheorem{ques}[thm]{Question}
\newtheorem{prob}[thm]{Problem}

\theoremstyle{remark}
\newtheorem{rmk}[thm]{Remark}

\newcommand{\N}{\mathbb{N}}
\newcommand{\Z}{\mathbb{Z}}

\newcommand{\R}{\mathbb{R}}
\newcommand{\C}{\mathbb{C}}
\newcommand{\Ha}{\mathbb{H}}
\newcommand{\K}{\mathbb{K}}

\newcommand{\mf}{\mathfrak}
\newcommand{\g}{\mathfrak{g}}
\newcommand{\h}{\mathfrak{h}}
\newcommand{\p}{\mathfrak{p}}

\newcommand{\bs}{\backslash}

\DeclareMathOperator{\GL}{GL}
\DeclareMathOperator{\SL}{SL}
\DeclareMathOperator{\SO}{SO}
\DeclareMathOperator{\SU}{SU}
\DeclareMathOperator{\Sp}{Sp}
\DeclareMathOperator{\M}{M}
\DeclareMathOperator{\Sym}{Sym}
\DeclareMathOperator{\U}{U}

\DeclareMathOperator{\Span}{Span}
\DeclareMathOperator{\Inn}{Inn}
\DeclareMathOperator{\cd}{cd}
\DeclareMathOperator{\Ad}{Ad}
\DeclareMathOperator{\ad}{ad}
\DeclareMathOperator{\trace}{trace}
\DeclareMathOperator{\diag}{diag}
\DeclareMathOperator{\Vol}{Vol}

\begin{document}

\title[Cartan projections and proper actions]
{Cartan projections of some non-reductive subgroups 
and proper actions on homogeneous spaces}
\author{Yosuke Morita}
\address{Department of Mathematics, Graduate School of Science,
Kyoto University, Kitashirakawa Oiwake-cho, Sakyo-ku,
Kyoto 606-8502, Japan}
\email{yosuke.m@math.kyoto-u.ac.jp}

\begin{abstract}
Kobayashi [\textit{Duke Math.\ J.}\ (1992)] gave a necessary condition 
for the existence of compact Clifford--Klein forms in terms of 
Cartan projections and non-compact dimensions of reductive subgroups. 
We extend his method to non-reductive subgroups, 
and give some examples of homogeneous spaces of reductive type 
that do not admit compact Clifford--Klein forms by comparing 
Cartan projections and non-compact dimensions of reductive subgroups 
with those of non-reductive subgroups. 
\end{abstract}

\maketitle

\section{Introduction}

\subsection{The problem and our approach}

Let $G$ be a Lie group and $H$ be a closed subgroup of $G$. 
If a discrete subgroup $\Gamma$ of $G$ acts properly and freely on 
$G/H$, the quotient space $\Gamma \bs G/H$ becomes a manifold 
locally modelled on $G/H$, and is called a \emph{Clifford--Klein form}.
When $H$ is compact 
(e.g.\ when $G/H$ is a Riemannian symmetric space of non-compact type), 
every discrete subgroup of $G$ acts properly on $G/H$, 
and the study of Clifford--Klein forms 
reduces to the study of discrete subgroups of $G$. 
However, when $H$ is non-compact, the study of Clifford--Klein forms 
has a very different nature, and is often much harder, 
because the properness of the action is not guaranteed. 

In this paper, we study the following problem: 

\begin{prob}\label{prob:existence}
Given a homogeneous space $G/H$, 
determine if it admits a \emph{compact} Clifford--Klein form or not. 
\end{prob}

Since the late 1980s, 
Problem~\ref{prob:existence} has been studied by many mathematicians, 
especially when $G/H$ is a homogeneous space of reductive type 
and $H$ is non-compact. 
A wide variety of methods have been applied to study this problem,
namely, the Cartan projection 
(e.g.\ \cite{Kob89}, \cite{Kob92}, \cite{Ben96}, \cite{Oku13}), 
Zimmer's cocycle superrigidity theorem
(e.g.\ \cite{Zim94}, \cite{LMZ95}, \cite{Lab-Zim95}), 
the estimate of matrix coefficients of 
infinite-dimensional unitary representations
(e.g.\ \cite{Mar97}, \cite{Sha00}), 
and the comparison of relative Lie algebra cohomology and 
de Rham cohomology 
(e.g.\ \cite{Kob-Ono90}, \cite{Ben-Lab92}, \cite{Mor15}, \cite{Tho15+}, 
\cite{Mor17Selecta}). 
These methods are applicable to partially overlapping but different 
examples. In this paper, we study Problem~\ref{prob:existence} 
by the Cartan projection method. More specifically, 
our approach is based on the following result of Kobayashi~\cite{Kob92}: 

\begin{fact}[{\cite[Th.~1.5]{Kob92}}; see Sections~\ref{sect:preliminaries-red} and 
\ref{sect:preliminaries-proper} for terminologies]
\label{fact:Kob92-1}
Let $G/H$ be a homogeneous space of reductive type. 
Let $\mu \colon G \to \mf{a}^+$ denote the Cartan projection of $G$. 
Given a Lie group $L$ with finitely many connected components, 
let $d(L)$ denote the non-compact dimension of $L$. 
Suppose that there exists a closed subgroup $H'$ reductive in $G$ 
such that 
\begin{itemize}
\item $\mu(H') \subset \mu(H)$.
\item $d(H') > d(H)$.
\end{itemize}
Then $G/H$ does not admit a compact Clifford--Klein form.
\end{fact}

Our starting point is to generalize Fact~\ref{fact:Kob92-1} 
in the following way: 

\begin{prop}\label{prop:Kob92-2}
Let $G/H$, $\mu$, and $d$ be as in Fact~\ref{fact:Kob92-1}. 
\begin{enumerate}[label = (\arabic*)]
\item Suppose that there exists a closed subgroup $H'$ 
with finitely many connected components (possibly not reductive in $G$) 
such that 
\begin{itemize}
\item $\mu(H') \subset \mu(H) + \mathcal{C}$ 
for some compact subset $\mathcal{C}$ of $\mf{a}$.
\item $d(H') > d(H)$.
\end{itemize}
Then $G/H$ does not admit a compact Clifford--Klein form.
\item Suppose that there exists a closed subgroup $H'$ 
with finitely many connected components (possibly not reductive in $G$) 
such that 
\begin{itemize}
\item $\mu(H') \subset \mu(H) + \mathcal{C}$ 
for some compact subset $\mathcal{C}$ of $\mf{a}$.
\item $d(H') = d(H)$.
\item $G/H'$ does not admit a compact Clifford--Klein form. 
\end{itemize}
Then $G/H$ does not admit a compact Clifford--Klein form.
\end{enumerate}
\end{prop}

The proof of Proposition~\ref{prop:Kob92-2} is similar to that of 
Fact~\ref{fact:Kob92-1}, although the author could not find the statement 
explicitly stated somewhere else. 

At a first glance, it is unclear if Proposition~\ref{prop:Kob92-2} 
is a meaningful generalization, 
because of the following two difficulties: 
\begin{itemize}
\item To apply Proposition~\ref{prop:Kob92-2}~(1) or (2), 
we need to find a closed subgroup $H'$ 
whose Cartan projection is contained in that of $H$ 
up to compact subsets. Unlike the reductive case, 
determining the Cartan projection of a non-reductive closed subgroup 
might be difficult. 
\item To apply Proposition~\ref{prop:Kob92-2}~(2), we need to choose 
$H'$ so that the non-existence of compact Clifford--Klein forms of 
$G/H'$ can be proved by some other method. 
\end{itemize}
However, these difficulties are resolved to some extent by 
the development of the study of Clifford--Klein forms: 
\begin{itemize}
\item Some non-reductive subgroups have been proved to have 
good Cartan projections 
(Kobayashi~\cite[Ex.~6.5]{Kob96}, 
Oh--Morris~\cite{Oh-Wit00}, \cite{Oh-Wit02}, \cite{Oh-Mor04}). 
\item Some methods to prove the non-existence of compact Clifford--Klein 
forms in the non-reductive setting have been found. 
Perhaps, these methods can be used to apply 
Proposition~\ref{prop:Kob92-2}~(2) to some concrete examples. 
\end{itemize}
These suggest the possibility that we can find some examples of $G/H$ 
for which Proposition~\ref{prop:Kob92-2} is applicable, while 
Fact~\ref{fact:Kob92-1} is not. 
We show in this paper that this is indeed the case. 

\subsection{Main results}

Now, let us state some results proved by the above idea.
First, we obtain the following result: 

\begin{thm}\label{thm:main-sl}
Let $\K = \R, \C$, or $\Ha$. 
If $m \geqslant 2$ and $n \geqslant 5m/4 + \varepsilon(m, \K)$, then 
$G/H = \SL(n, \K) / \SL(m, \K)$ 
does not admit a compact Clifford--Klein form. Here, the number 
$\varepsilon(m, \K)$ is defined by: 
\begin{align*}
\varepsilon(m, \R) &= 
\begin{cases}
1 & \text{if $m \equiv 0 \pmod 4$}, \\
7/4 & \text{if $m \equiv 1 \pmod 4$}, \\
1/2 & \text{if $m \equiv 2 \pmod 4$}, \\
9/4 & \text{if $m \equiv 3 \pmod 4$}, 
\end{cases} \\
\varepsilon(m, \C) &= 
\begin{cases}
0 & \text{if $m \equiv 0 \pmod 4$}, \\
7/4 & \text{if $m \equiv 1 \pmod 4$}, \\
1/2 & \text{if $m \equiv 2 \pmod 4$}, \\
5/4 & \text{if $m \equiv 3 \pmod 4$}, 
\end{cases} \\
\varepsilon(m, \Ha) &= 
\begin{cases}
0 & \text{if $m \equiv 0 \pmod 4$}, \\
3/4 & \text{if $m \equiv 1 \pmod 4$}, \\
1/2 & \text{if $m \equiv 2 \pmod 4$}, \\
5/4 & \text{if $m \equiv 3 \pmod 4$}. 
\end{cases}
\end{align*}
\end{thm}

\begin{rmk}[cf.\ Remark~\ref{rmk:sl-other}]
Unfortunately, when $\K = \R$ or $\C$, Theorem~\ref{thm:main-sl} 
is not as sharp as the result of Labourie--Zimmer~\cite{Lab-Zim95} 
for a large $m$, although our method is, 
as far as the author understands, completely different from theirs. 
\end{rmk}

Also, we obtain a new proof of the following result: 

\begin{thm}[{\cite[Cor.~6.2]{Mor17Selecta}}]\label{thm:main-associated}
Let $G/H$ be a reductive symmetric space and 
$(\g, \h)$ be the corresponding reductive symmetric pair. Assume that: 
\begin{itemize}
\item $H$ is non-compact. 
\item The associated symmetric subalgebra $\h^a = \mf{k} \cap \h \oplus \p \cap \mf{q}$ (cf.~Definition~\ref{defn:associated}) 
is a Levi subalgebra of $\g$. 
\end{itemize}
Then, $G/H$ does not admit a compact Clifford--Klein form. 
\end{thm}

\begin{ex}\label{ex:sl-so}
If $p,q \geqslant 1$, then 
$\SL(p+q, \R)/\SO_0(p,q)$ does not admit a compact Clifford--Klein form. 
\end{ex}

\begin{rmk}\label{rmk:main-associated-first}
The proof of Theorem~\ref{thm:main-associated} that we gave in \cite{Mor17Selecta} 
is based on the comparison of relative Lie algebra cohomology and 
de Rham cohomology. 
See \cite[Table 3]{Mor17Selecta} for the list of 
irreducible symmetric spaces for which Theorem~\ref{thm:main-associated} applies. 
This list is obtained from 
Lemma~\ref{lem:symmetric-levi}~(2) (i)~$\Leftrightarrow$~(v). 
\end{rmk}

\subsection{Related previous results}

We mention some previous results related to our main results. 

\begin{rmk}\label{rmk:sl-other}
Let $n > m \geqslant 2$.
Let us summarize previous results on the non-existence of 
compact Clifford--Klein forms of 
$\SL(n, \K) / \SL(m, \K)$ ($\K = \R, \C$, or $\Ha$): 
\begin{enumerate}[label = (\roman*)]
\item For $m \geqslant 2$, 
define $\delta(m, \K)$ as follows: 
\begin{align*}
\delta(m, \R) &= 
\begin{cases}
1 & \text{if $m$: even}, \\
5/2 & \text{if $m$: odd},
\end{cases} \\
\delta(m, \C) &= 
\begin{cases}
0 & \text{if $m$: even}, \\
3/2 & \text{if $m$: odd},
\end{cases} \\
\delta(m, \Ha) &= 
\begin{cases}
0 & \text{if $m$: even}, \\
1/2 & \text{if $m$: odd}.
\end{cases}
\end{align*}
Using Fact~\ref{fact:Kob92-1}, 
Kobayashi~\cite{Kob92}, \cite{Kob97} proved that
$\SL(n, \K) / \SL(m, \K)$ does not admit a compact Clifford--Klein form 
when $n \geqslant 3m/2 + \delta(m, \K)$. 
\item Applying Zimmer's cocycle superrigidity theorem, 
Zimmer~\cite{Zim94}, Labourie--Mozes--Zimmer~\cite{LMZ95} proved that 
$\SL(n, \K)/\SL(m, \K)$ does not admit a compact Clifford--Klein form 
when $n \geqslant \max \{ 2m, 5 \}$. 
\item Improving the results in \cite{Zim94} and \cite{LMZ95}, 
Labourie--Zimmer~\cite{Lab-Zim95} proved that 
$\SL(n, \K)/\SL(m, \K)$ does not admit a compact Clifford--Klein form 
when $n \geqslant m+3$ and $\K = \R$ or $\C$. 
As far as the author understands, although they gave a proof only for 
$\K = \R$, it equally works for $\K = \C$, while not for $\K = \Ha$.
\item Using Cartan projections (while not Fact~\ref{fact:Kob92-1}), 
Benoist~\cite{Ben96} proved that $\SL(m+1, \K) / \SL(m, \K)$ 
does not admit a compact Clifford--Klein form when $m$ is even. 
\item Applying the estimate of matrix coefficients, Shalom~\cite{Sha00} 
proved that $\SL(n, \K) / \SL(2, \K)$ 
does not admit a compact Clifford--Klein form when $n \geqslant 4$. 
\item Comparing relative Lie algebra cohomology and de Rham cohomology, 
Tholozan~\cite{Tho15+} and the author \cite{Mor17Selecta} proved that 
$\SL(n, \R) / \SL(m, \R)$ does not admit a compact Clifford--Klein form 
when $m$ is even. Note that their proofs do not work for 
$\K = \C$ or $\Ha$. 
\end{enumerate}

Thus, 
\begin{itemize}
\item When $\K = \R$, Theorem~\ref{thm:main-sl} is covered by previous works. 
\item When $\K = \C$, our result seems new for $(n,m) = (5,3), (8,6), (10,8)$. 
\item When $\K = \Ha$, our result seems new for infinitely many $(n, m)$. 
\end{itemize}

There are two other closely related results: 
\begin{enumerate}[label = (\roman*)]
\item[(vii)] For $n \geqslant 1$, let 
$\alpha_n \colon \SL(2, \R) \to \SL(n, \R)$ 
be the $n$-dimensional irreducible real representation of $\SL(2, \R)$. 
Applying the estimate of matrix coefficients, 
Margulis~\cite{Mar97} proved that $\SL(n, \R) / \alpha_n(\SL(2, \R))$ 
does not admit a compact Clifford--Klein form when $n \geqslant 4$. 
\item[(viii)] 
Given a reductive Lie group $G$ with Lie algebra $\g$ and 
Cartan involution $\theta$, 
let $K = G^\theta$ and $\p = \g^{-\theta}$ 
be the maximal compact subgroup of $G$ 
and the $(-1)$-eigenspace of $\theta$ in $\g$, respectively, 
and define the \emph{Cartan motion group} 
$G_\theta$ associated with $G$ by $G_\theta = K \ltimes \p$. 
For $m \geqslant 2$, let $\varepsilon(m, \K)$ be as in 
Theorem~\ref{thm:main-sl}, and define $\varepsilon'(m, \K)$ as follows: 
\begin{align*}
\varepsilon'(m, \R) &= 
\begin{cases}
\varepsilon(m, \R) + 1 & \text{if $m \equiv 2 \pmod 4$}, \\
\varepsilon(m, \R) & \text{otherwise}, 
\end{cases} \\
\varepsilon'(m, \C) &= 
\begin{cases}
\varepsilon(m, \C) + 1 & \text{if $m \equiv 0 \pmod 4$}, \\
\varepsilon(m, \C) & \text{otherwise}, 
\end{cases} \\
\varepsilon'(m, \Ha) &= \varepsilon(m, \Ha).
\end{align*}
Using an analogue of Fact~\ref{fact:Kob92-1} for Cartan motion groups 
(\cite[Th.~9]{Yos07IJM}), Yoshino~\cite{Yos07RepSymp} proved that
$\SL(n, \K)_\theta / \SL(m, \K)_\theta$ 
does not admit a compact Clifford--Klein form 
when $n \geqslant 5m/4 + \varepsilon'(m, \K)$.
\end{enumerate}
Our proof of Theorem~\ref{thm:main-sl} is inspired by (i) and (vii). 
See Remarks~\ref{rmk:compare-Kob92} and \ref{rmk:compare-Yos07RepSymp}. 
\end{rmk}

\begin{rmk}
Although Theorem~\ref{thm:main-associated} was first proved in 
\cite[Cor.~6.2]{Mor17Selecta}, 
some special cases had been proved before that 
(cf.\ Kobayashi~\cite{Kob89}, \cite{Kob92}, 
Kobayashi--Ono~\cite{Kob-Ono90}, Benoist--Labourie~\cite{Ben-Lab92}, 
Benoist~\cite{Ben96}, Okuda~\cite{Oku13}, the author~\cite{Mor15}).
For example, the non-existence of compact Clifford--Klein forms of 
$\SL(p+q, \R) / \SO_0(p,q) \ (p,q \geqslant 1)$ 
had been known before \cite{Mor17Selecta} in some cases: 
\begin{enumerate}[label = (\roman*)]
\item Using Fact~\ref{fact:Kob92-1}, 
Kobayashi~\cite{Kob92} solved the case of $p=q$. 
\item Using Cartan projections (while not Fact~\ref{fact:Kob92-1}), 
Benoist~\cite{Ben96} solved the case of $p=q$ or $q+1$. 
\item Comparing relative Lie algebra cohomology and de Rham cohomology, 
the author~\cite{Mor15} solved the case where $p$ and $q$ are odd. 
\end{enumerate}
\end{rmk}

\subsection{Outline of the paper}\label{subsect:outline}

In Section~\ref{sect:preliminaries-red}, 
we recall basic definitions in the theory of reductive Lie groups, and 
prove some lemmas needed in Section~\ref{sect:proof-main-associated}.
In Section~\ref{sect:preliminaries-proper}, 
we recall basic definitions and results on 
proper actions on homogeneous spaces. 
In Sections~\ref{sect:proof-Kob92-2}, \ref{sect:proof-main-sl}, 
and \ref{sect:proof-main-associated},
we prove Proposition~\ref{prop:Kob92-2}, Theorem~\ref{thm:main-sl}, 
and Theorem~\ref{thm:main-associated}, respectively. 
In Section~\ref{sect:proof-main-associated}, 
we first prove Theorem~\ref{thm:main-associated} in a special case, 
$G/H = \SL(p+q, \R)/\SO_0(p,q)$, 
so as to clarify the strategy of our proof. We then give a proof of the general case. 
Section~\ref{sect:proof-main-associated} 
can be read independently from Section~\ref{sect:proof-main-sl}. 
Finally, we give some remarks and mention some open questions 
raised by this work in Section~\ref{sect:remarks-questions}. 

\section{Preliminaries on Reductive Lie groups}
\label{sect:preliminaries-red}

\subsection{Structure of reductive Lie groups}\label{subsect:real-red}

Let $G$ be a reductive Lie group with Lie algebra $\g$ 
(as usual, we use uppercase Roman letters for Lie groups and 
lowercase German letters for the corresponding Lie algebras). 
Fix a Cartan involution $\theta$ on $G$. 
Let $K = G^\theta$ be the maximal compact subgroup of $G$, 
and let $\g = \mf{k} \oplus \p$ be the eigenspace decomposition of $\g$ 
with respect to $\theta$. Note that $\mf{k}$ is the Lie algebra of $K$. 
Fix a maximal abelian subspace $\mf{a}$ of $\p$ 
(called the \emph{Cartan subspace} or 
the \emph{maximal split abelian subspace} of $\g$), and put
\[
\g_\lambda = 
\{ X \in \g \mid [Z, X] = \lambda(Z) X \ \text{for any} \ Z \in \mf{a} \}
\]
for each $\lambda \in \mf{a}^\ast$. Let 
\[
\g = \g_0 \oplus \bigoplus_{\lambda \in \Sigma} \g_\lambda, \qquad 
\Sigma = \{ \lambda \in \mf{a}^\ast \mid 
\g_\lambda \neq 0 \} \smallsetminus \{ 0 \}
\]
be the restricted root-space decomposition of $\g$ relative to $\mf{a}$. 
Fix a simple system $\Pi$ of the restricted root system $\Sigma$. 
Let us define a partial order $\leqslant_\Pi$ 
on $\Sigma \sqcup \{ 0 \}$ as follows: 
\begin{itemize}
\item $\lambda \leqslant_\Pi \lambda'$ holds if and only if 
$\lambda' - \lambda$ is written as a linear combination of 
elements of $\Pi$ with non-negative integer coefficients. 
\end{itemize}
There exists a unique maximal element 
$\widetilde{\lambda}$ of $\Sigma \sqcup \{ 0 \}$ 
with respect to the $\leqslant_\Pi$, 
called the \emph{highest root} of $\Sigma$ relative to $\Pi$. 
Let $\mf{a}^+ \subset \mf{a}$ be the closed positive Weyl chamber 
with respect to $\Pi$. 
The $KAK$ decomposition theorem (see e.g.\ \cite[Th.~7.39]{Kna02}) 
says that every element $g \in G$ may be written as $g = k_1\exp(X)k_2$ 
for some $k_1, k_2 \in K$ and some $X \in \mf{a}^+$, 
and the choice of such $X$ is unique. 

\begin{defn}
We define the \emph{Cartan projection} $\mu \colon G \to \mf{a}^+$ by 
$\mu(g) = X$, which is continuous and proper. 
\end{defn}

We fix a $G$-invariant and $\theta$-invariant 
non-degenerate symmetric bilinear form $B$ on $\g$ such that
\[
B_\theta \colon \g \times \g \to \R, \qquad 
(X, Y) \mapsto -B(X, \theta Y)
\]
is positive definite. If $G$ is semisimple, one may take $B$ to be the Killing form 
\[
B(X, Y) = \trace(\ad(X) \ad(Y)). 
\]
We write $\|{-}\|_\theta$ for the norm associated with $B_\theta$. 

\subsection{Levi subalgebras and horospherical subalgebras}
\label{subsect:levi-horospherical}

Let $\Sigma_+$ and $\Sigma_-$ denote the sets of positive and negative 
roots in $\Sigma$ relative to the simple system $\Pi$, respectively. 
Given a subset $\Pi'$ of $\Pi$, decompose $\Sigma$ as 
\[
\Sigma = 
\Sigma_{\Pi', +} \sqcup \Sigma_{\Pi', 0} \sqcup \Sigma_{\Pi', -}, 
\]
where 
\[
\Sigma_{\Pi', 0} = \Sigma \cap \Span_\Z \Pi', \qquad 
\Sigma_{\Pi', +} = \Sigma_+ \smallsetminus \Sigma_{\Pi', 0}, \qquad 
\Sigma_{\Pi', -} = \Sigma_- \smallsetminus \Sigma_{\Pi', 0}.
\]
\begin{defn}
We define the \emph{Levi subalgebra} $\mf{l}_{\Pi'}$ and
the \emph{horospherical subalgebra} $\mf{u}_{\Pi'}$ of $\g$
associated with $\Pi'$ by
\[
\mf{l}_{\Pi'} = 
\g_0 \oplus \bigoplus_{\lambda \in \Sigma_{\Pi', 0}} \g_{\lambda}, 
\qquad 
\mf{u}_{\Pi'} = 
\bigoplus_{\lambda \in \Sigma_{\Pi', +}} \g_\lambda. \]
\end{defn}

We say that a Lie subalgebra of $\g$ is a \emph{Levi} 
(resp.\ \emph{horospherical}) \emph{subalgebra} when it is conjugate to 
$\mf{l}_{\Pi'}$ (resp.\ $\mf{u}_{\Pi'}$) for some $\Pi' \subset \Pi$. 
A horospherical subalgebra is normalized by 
the corresponding Levi subalgebra. 

Let $\mf{u}$ be a horospherical subalgebra of $\g$. 
The connected subgroup $U$ of $G$ corresponding to $\mf{u}$ 
is called the \emph{horospherical subgroup}. 
It follows from the Iwasawa decomposition that $U$ is closed in $G$ 
and the exponential map $\exp \colon \mf{u} \to U$ is a diffeomorphism. 

\subsection{Reductive symmetric spaces}
\label{subsect:symmetric}

\begin{defn}
\begin{enumerate}[label = (\arabic*)]
\item A subalgebra $\h$ of $\g$ is called \emph{symmetric} if
there exists an involution $\sigma$ of $\g$ such that
$\sigma \theta = \theta \sigma$ and $\h = \g^\sigma$, 
where $\g^\sigma$ is the fixed-point subalgebra of $\sigma$. 
We then say that $(\g, \h)$ is a \emph{reductive symmetric pair}.
\item A closed subgroup $H$ of $G$ is called \emph{symmetric} if 
there exists an involution $\sigma$ of $G$ such that 
$\sigma \theta = \theta \sigma$ and 
$H$ is an open subgroup of the fixed-point group $G^\sigma$.
We then say that $G/H$ is a \emph{reductive symmetric space}.
\end{enumerate}
\end{defn}

\begin{rmk}
The condition $\sigma \theta = \theta \sigma$ 
is a mild assumption; if $\g$ is a semisimple Lie algebra, 
for any involution $\sigma$ of $\g$, 
there exists a Cartan involution $\theta$ of $\g$ such that
$\sigma \theta = \theta \sigma$ (\cite[Lem.~10.2]{Ber57}), 
and the choice of such $\theta$ is unique up to conjugation by 
$\Inn(\h)$ (\cite[Lem.~4]{Mat79}). 
\end{rmk}

\begin{rmk}\label{rmk:involution-determined}
Suppose that $\h$ is a symmetric subalgebra of $\g$ defined by 
an involution $\sigma$. 
Let $\mf{q}$ denote the orthogonal complement of $\h$ 
with respect to $B$. Then, one can easily see that 
$\mf{q}$ is the $(-1)$-eigenspace of $\sigma$, 
and therefore, $\g = \h \oplus \mf{q}$.
This implies that $\sigma$ is uniquely determined by $\h$.
\end{rmk}

\begin{defn}\label{defn:associated}
Let $(\g, \h)$ be a symmetric subalgebra 
determined by an involution $\theta$, and let $\mf{q}$ 
be the orthogonal complement of $\h$ with respect to $B$. 
We define the \emph{associated symmetric subalgebra} 
$\h^a$ of $(\g, \h)$ to be the symmetric subalgebra defined by 
the involution $\sigma\theta$. 
In other words, $\h^a = \mf{k} \cap \h \oplus \p \cap \mf{q}$.
We call $(\g, \h^a)$ the \emph{associated symmetric pair} of $(\g, \h)$.
\end{defn}

\subsection{Some lemmas}

The following lemmas are needed in 
Section~\ref{sect:proof-main-associated}: 

\begin{lem}\label{lem:in-p}
Let $\lambda \in \Sigma$ and $X, X' \in \g_\lambda$. 
\begin{enumerate}[label = (\arabic*)]
\item If $[X, \theta X'] = 0$, then either $X$ or $X'$ is zero. 
\item If $[X, \theta X'] \in \p$, 
then $X$ and $X'$ are linearly dependent.
\end{enumerate}
\end{lem}
\begin{proof}
(1) Assume that $X$ is non-zero. 
The linear map
$i \colon \mf{sl}(2, \R) \to \g$ defined by 
\[
i \begin{pmatrix} 0 & 1 \\ 0 & 0 \end{pmatrix} = X, \quad \ 
i \begin{pmatrix} 0 & 0 \\ 1 & 0 \end{pmatrix} = -\frac{2}{\| H_\lambda \|_\theta^2 \| X \|_\theta^2} \theta X, \quad \ 
i \begin{pmatrix} 1 & 0 \\ 0 & -1 \end{pmatrix} = \frac{2}{\| H_\lambda \|_\theta^2} H_\lambda 
\]
is a Lie algebra homomorphism, 
where $H_\lambda$ is the element of $\mf{a}$ corresponding to $\lambda$ 
via $B_\theta$ (see e.g.\ \cite[Prop.~6.52]{Kna02}). Then, 
$\bigoplus_{n \in \Z} \g_{n\lambda}$ 
is an $\mf{sl}(2, \R)$-module via $i$, and 
$\g_{n \lambda}$ is a weight space with weight $2n$ for each $n \in \Z$. 
It follows from the representation theory of $\mf{sl}(2, \R)$ that 
$\ad(X) \colon \g_{-\lambda} \to \g_0$ is injective. 
Therefore, $[X, \theta X'] = 0$ only if $X' = 0$. 

(2) Assume that $X$ is non-zero. 
Notice that $[X, \theta X']$ and $[X, \theta X]$ are elements of 
$\mf{a} \ (= \g_0 \cap \p)$. 
For each $Z \in \mf{a}$, 
\begin{align*}
B_\theta(Z, [X, \theta X']) 
&= B([Z, X], \theta X')
= \lambda(Z) B(X, \theta X') \\
&= B_\theta(Z, -B_\theta(X, X')H_\lambda). 
\end{align*}
Therefore, $[X, \theta X'] = -B_\theta(X, X')H_\lambda$. Similarly, 
$[X, \theta X] = - \| X \|_\theta^2 H_\lambda$. We see that 
\[
[X, \theta(B_\theta(X, X')X - \| X \|_\theta^2 X') ] = 0.
\]
By (1), $B_\theta(X, X')X - \| X \|_\theta^2 X' = 0$. 
\end{proof}

\begin{lem}\label{lem:symmetric-levi}
\begin{enumerate}[label = (\arabic*)]
\item The following conditions on $\Pi'$ are all equivalent: 
\begin{enumerate}[label = (\roman*)]
\item $(\g, \mf{l}_{\Pi'})$ is a reductive symmetric pair.
\item The linear involution $\sigma \colon \g \to \g$ defined by 
\[
\sigma|_{\mf{l}_{\Pi'}} = 1, \qquad 
\sigma|_{\mf{u}_{\Pi'} \oplus \theta \mf{u}_{\Pi'}} = -1
\]
is compatible with the Lie algebra structure, i.e.\ 
$[\sigma X, \sigma Y] = \sigma [X, Y]$ for any $X, Y \in \g$. 
\item $\mf{u}_{\Pi'}$ is abelian.
\item For any $\lambda, \lambda' \in \Sigma_{\Pi', +}$, 
$\lambda + \lambda' \notin \Sigma$. 
\end{enumerate}
\item Suppose that $\g$ is simple, and let $\widetilde{\lambda}$ 
denote the highest root of $\Sigma$ relative to $\Pi$. 
Then, the above conditions (i)--(iv) are also equivalent to: 
\begin{enumerate}
\item[(v)] Either $\Pi' = \Pi$ or there exists $\alpha \in \Pi$ 
satisfying $\Pi' = \Pi \smallsetminus \{ \alpha \}$ and 
$\widetilde{\lambda} - \alpha \in \Sigma_{\Pi', 0}$.
\end{enumerate}
\end{enumerate}
\end{lem}
\begin{proof}
(1) (ii)~$\Rightarrow$~(i): 
Obvious. 

(i)~$\Rightarrow$~(ii): 
The orthogonal complement of $\mf{l}_{\Pi'}$ in $\g$ 
with respect to $B$ is $\mf{u}_{\Pi'} \oplus \theta(\mf{u}_{\Pi'})$. 
It follows from Remark~\ref{rmk:involution-determined} that $\sigma$ 
is the involution of Lie algebras that defines $\mf{l}_{\Pi'}$. 

(ii)~$\Rightarrow$~(iii): 
Take any $X, Y \in \mf{u}_{\Pi'}$. Since $[X, Y] \in \mf{u}_{\Pi'}$, 
we see that 
\[
[X, Y] = [\sigma X, \sigma Y] = \sigma [X, Y] = -[X, Y], 
\]
i.e.\ $[X, Y] = 0$. Thus, $\mf{u}_{\Pi'}$ is abelian. 

(iii)~$\Rightarrow$~(iv): 
Suppose that (iv) is not satisfied, and take 
$\lambda, \lambda' \in \Sigma_{\Pi', +}$ such that 
$\lambda + \lambda' \in \Sigma$. By \cite[Lem.~7.75]{Kna02}, there exists 
$X \in \g_{\lambda} \ (\subset \mf{u}_{\Pi'})$ and 
$Y \in \g_{\lambda'} \ (\subset \mf{u}_{\Pi'})$ such that $[X, Y] \neq 0$. 
This means that (iii) is not satisfied. 

(iv)~$\Rightarrow$~(ii): 
To see that $\sigma$ is an involution of Lie algebras, 
it is enough to verify that 
\[
[\mf{u}_{\Pi'}, \mf{u}_{\Pi'}] = \{ 0 \}, \qquad 
[\mf{u}_{\Pi'}, \theta \mf{u}_{\Pi'}] \subset \mf{l}_{\Pi'}. 
\]
The first equality directly follows from condition (iv). 
To prove the second equality, it suffices to see that, 
if $\lambda, \lambda' \in \Sigma_{\Pi', +}$, 
then either $\lambda - \lambda' \in \Sigma_{\Pi', 0}$ or 
$\lambda - \lambda' \notin \Sigma$ holds. 
Suppose the contrary, namely, suppose that there exist 
$\lambda, \lambda' \in \Sigma_{\Pi', +}$ such that 
$\lambda - \lambda' \in \Sigma_{\Pi', +} \sqcup \Sigma_{\Pi', -}$. 
If $\lambda - \lambda' \in \Sigma_{\Pi', +}$, then 
$\lambda', \lambda - \lambda' \in \Sigma_{\Pi', +}$ and 
$\lambda' + (\lambda - \lambda') \in \Sigma$. This contradicts to 
condition (iv). 
Similarly, if $\lambda - \lambda' \in \Sigma_{\Pi', -}$, 
then $\lambda, \lambda' - \lambda \in \Sigma_{\Pi', +}$ and 
$\lambda + (\lambda' - \lambda) \in \Sigma$, 
which is a contradiction. 

(2) (v)~$\Rightarrow$~(iv): 
Obvious. 

(iv)~$\Rightarrow$~(v): 
Let us take $\alpha_1, \dots, \alpha_n \in \Pi$ so that 
$\sum_{i=1}^m \alpha_i \in \Sigma$ for every $1 \leqslant m \leqslant n$ 
and $\sum_{i=1}^n \alpha_i = \widetilde{\lambda}$ 
(see e.g.\ \cite[Ch.~VI, \S 1, Prop.~19]{BouLie4-6}). 
Suppose that (v) is not satisfied, and take 
$1 \leqslant k < \ell \leqslant n$ such that 
$\alpha_k, \alpha_\ell \in \Pi'$. 
Then, $\sum_{i=1}^{\ell-1} \alpha_i, \alpha_\ell \in \Sigma_{\Pi', +}$ 
and $(\sum_{i=1}^{\ell-1} \alpha_i) + \alpha_\ell \in \Sigma$, 
i.e.\ (iv) is not satisfied. 
\end{proof}

\section{Preliminaries on proper actions on homogeneous spaces}
\label{sect:preliminaries-proper}

\subsection{The symbols $\prec$, $\sim$, and $\pitchfork$}

\begin{defn}
Let $G$ be a locally compact group. 
Let $H$ and $L$ be two closed subsets of $G$.
\begin{enumerate}[label = (\arabic*)]
\item We say that \emph{$H \prec L$ in $G$} if there exists 
a compact subset $C$ of $G$ such that $H \subset CLC^{-1}$.
If both $H \prec L$ and $L \prec H$ hold in $G$, 
we say that \emph{$H \sim L$ in $G$}.
\item We say that \emph{$H \pitchfork L$ in $G$} if
$CHC^{-1} \cap L$ is compact for every compact subset $C$ of $G$.
Here, $CHC^{-1} = \{ c_1 h c_2^{-1} \mid c_1, c_2 \in C,\ h \in H \}$. 
\end{enumerate}
\end{defn}

\begin{rmk}
The symbols $\sim$ and $\pitchfork$ are due to Kobayashi~\cite{Kob96}. 
Kobayashi used the terminology \emph{`$H \subset L$ modulo $\sim$'} 
for $H \prec L$. Benoist~\cite{Ben96} used the terminologies
\emph{`$H$ is contained in $L$ modulo the compacts of $G$'} and
\emph{`$(H, L)$ is $G$-proper'} for $H \prec L$ and $H \pitchfork L$, 
respectively.
\end{rmk}

One can easily see that: 

\begin{fact}[cf.\ {\cite[\S 3.1]{Ben96}, \cite[\S 2]{Kob96}}]
\label{fact:proper-action-general}
Let $G$ be a locally compact group.
\begin{enumerate}[label = (\arabic*)]
\item $\prec$ is a preorder 
(and therefore, $\sim$ is an equivalence relation). 
\item $\pitchfork$ is a symmetric relation. 
\item Let $H$ and $L$ be two closed subgroups of $G$.
Then, $L$ acts properly on $G/H$ if and only if
$H \pitchfork L$ in $G$. 
\item Let $H, H'$, and $L$ be three closed subsets of $G$.
If $H' \prec H$ and $H \pitchfork L$ in $G$, 
then $H' \pitchfork L$ in $G$.
\item Let $\widetilde{G}$ be a locally compact group containing $G$ 
as a closed subgroup. 
Let $H$ and $H'$ be two closed subsets of $G$. 
Let $L$ be a closed subset of $\widetilde{G}$ such that 
$G \cap L = \{ 1 \}$ and $L \subset Z_{\widetilde{G}}(G)$. 
If $H' \prec H$ in $G$, 
then $H' \times L \prec H \times L$ in $\widetilde{G}$. 
\item Suppose that $G$ is a reductive Lie group and let 
$\mu \colon G \to \mf{a}^+$ be the Cartan projection. 
Let $H$ be a closed subset of $G$. 
Then, $H \sim \exp(\mu(H) + \mathcal{C})$ in $G$ 
for any compact subset $\mathcal{C}$ of $\mf{a}$. 
\end{enumerate}
\end{fact}

Let $G$ be a reductive Lie group. 
Let $H$ and $L$ be two closed subgroups of $G$. 
Then, it follows from 
Fact~\ref{fact:proper-action-general}~(3), (4), and (6) that 
the properness of the action of $L$ on $G/H$ is completely determined by 
the Cartan projections $\mu(H)$ and $\mu(L)$. 

\begin{rmk}
In fact, the following more precise result is known, 
although we do not use it in this paper: 
\end{rmk}

\begin{fact}[{\cite[Th.~5.1]{Ben96}, \cite[Th.~3.4]{Kob96}}]
Let $G$ be a reductive Lie group and 
$\mu \colon G \to \mf{a}^+$ be the Cartan projection. 
Let $H$ and $L$ be two closed subsets of $G$. Then, 
$H \pitchfork L$ holds if and only if 
$(\mu(H) + \mathcal{C}) \cap \mu(L)$ 
is compact for any compact subset $\mathcal{C}$ of $\mf{a}$. 
\end{fact}

\subsection{Cohomological dimensions and Clifford--Klein forms}

Recall that any Lie group with finitely many connected components
has a maximal compact subgroup unique up to conjugation
(see e.g.\ \cite[Ch.~XV, Th.~3.1]{Hoc65}, 
\cite[Ch.~VII, Th.~1.2]{Bor98}).
The following definition is from Kobayashi~\cite{Kob89}: 

\begin{defn}[{cf.\ \cite[\S 2]{Kob89}}]
For a Lie group $G$ with finitely many connected components, we put
$d(G) = \dim G - \dim K$ and call it the \emph{non-compact dimension} 
of $G$, where $K$ is a maximal compact subgroup of $G$. 
\end{defn}

\begin{defn}
For a discrete group $\Gamma$ and a commutative ring $A$, 
we define the \emph{$A$-cohomological dimension} 
$\cd_A \Gamma \in \N \sqcup \{ +\infty \}$ by 
\[
\cd_A \Gamma = \sup \{ n \in \N \mid 
\text{$H^n(\Gamma; V) \neq 0$ for some $A\Gamma$-module $V$} \}.
\]
\end{defn}

\begin{fact}
[cf.\ {\cite[Cor.~5.5]{Kob89}}, {\cite[Lem.~2.2]{Mor17Selecta}}]
\label{fact:cohomological-dimension}
Let $G$ be a Lie group and $H$ be a closed subgroup of $G$. 
Assume that $G$ and $H$ have finitely many connected components.
Let $\Gamma$ be a discrete subgroup of $G$ 
acting properly and freely on $G/H$ and 
$k$ be a field of characteristic zero. 
Then, the inequality 
\[
\cd_k \Gamma \leqslant d(G) - d(H)
\]
holds; the equality is attained 
if and only if the Clifford--Klein form $\Gamma \bs G/H$ is compact. 
\end{fact}

\subsection{An obstruction to the existence of 
compact Clifford--Klein forms}

Later, we use the following result: 

\begin{prop}[{\cite[Prop.~4.1]{Mor17PRIMS}}]
\label{prop:trace-free}
Let $G$ be a Lie group and 
$H$ be a closed subgroup of $G$ with finitely many connected components. 
Let $\mf{n}_\g(\h)$ denote the normalizer of $\h$ in $\g$. 
If the $\h$-action on $\g/\h$ is trace-free 
(i.e.\ $\trace(\ad_{\g/\h}(X)) = 0$ for all $X \in \h$) 
and the $\mf{n}_\g(\h)$-action on $\g/\h$ is not trace-free, 
then $G/H$ does not admit a compact Clifford--Klein form. 
\end{prop}

In a previous paper \cite{Mor17PRIMS}, we gave a proof of 
Proposition~\ref{prop:trace-free} based on the comparison of 
relative Lie algebra cohomology and de Rham cohomology. 
Here, we give another proof following the idea of 
Zimmer~\cite[Prop.~2.1]{Zim94}: 

\begin{proof}[Alternative proof of Proposition~\ref{prop:trace-free}]
Let $\Gamma \bs G/H$ be a Clifford--Klein form. 
Replacing $H$ with its identity component, 
we may assume that $H$ is connected. 
Let $\Omega^N(G/H)^G$ be the space of all $G$-invariant $N$-forms on $G/H$, where 
$N = \dim (G/H)$. 
Since the canonical map
\[
\Omega^N(G/H)^G \to (\Lambda^N (\g/\h))^\h, \qquad \omega \mapsto \omega_{1 \cdot H}
\]
is a linear isomorphism and since the $\h$-action on $\g/\h$ is trace-free, 
there is a non-trivial $G$-invariant volume form $\omega$ on $G/H$. 
Let $\omega_\Gamma$ be the associated volume form on $\Gamma \bs G/H$. 
We write $\Vol(\Gamma \bs G/H) \in (0, +\infty]$ 
for the volume of $\Gamma \bs G/H$ with respect to $\omega_\Gamma$. 
Take $X_0 \in \mf{n}_\g(\h)$ such that $\trace(\ad_{\g/\h}(X_0)) = 1$. 
For each $t \in \R$,  
\[
\varphi_t \colon \Gamma \bs G/H \to \Gamma \bs G/H, \qquad 
\Gamma gH \mapsto \Gamma g \exp(-tX_0) H
\]
is a well-defined diffeomorphism because $X_0$ belongs to the normalizer of $\h$. 
We have $\varphi_t^\ast \omega_\Gamma = e^t \omega_\Gamma$ and, therefore, 
\[
\Vol(\Gamma \bs G/H) = \int_{\Gamma \bs G/H} \omega_\Gamma 
= \int_{\Gamma \bs G/H} \varphi_t^\ast \omega_\Gamma 
= e^t \Vol(\Gamma \bs G/H). 
\]
This implies $\Vol(\Gamma \bs G/H) = +\infty$, 
hence $\Gamma \bs G/H$ cannot be compact. 
\end{proof}

\begin{rmk}
The existence problem of compact Clifford--Klein forms (Problem~\ref{prob:existence}) may be generalized in the following two ways: 
\begin{itemize}
\item Is there a (possibly non-compact) finite-volume Clifford--Klein form of $G/H$?
\item Is there a compact manifold locally modelled on $G/H$?
\end{itemize}
Here, by a \emph{manifold locally modelled on $G/H$}, we mean a manifold obtained by patching open subsets of $G/H$ by left translations by elements of $G$. A manifold locally modelled on $G/H$ is more general than Clifford--Klein forms. 

It follows from the above proof that, 
under the assumptions of Proposition~\ref{prop:trace-free}, 
there is no (compact or non-compact) finite-volume Clifford--Klein form of $G/H$. 
However, it cannot be used to show the non-existence of compact manifold locally modelled on $G/H$. 
In contrast, the proof given in \cite{Mor17PRIMS}
shows the non-existence of compact manifolds locally modelled on $G/H$, while it cannot be used to show the non-existence of non-compact finite-volume Clifford--Klein form. 
\end{rmk}

\section{Proof of Proposition~\ref{prop:Kob92-2}}
\label{sect:proof-Kob92-2}

By Fact~\ref{fact:proper-action-general}~(1) and (6), to prove 
Proposition~\ref{prop:Kob92-2}, it is enough to see the following: 

\begin{lem}\label{lem:Kob92-refine}
Let $G$ be a linear Lie group and $H$ be a closed subgroup of $G$. 
Assume that $G$ and $H$ have finitely many connected components. 
\begin{enumerate}[label = (\arabic*)]
\item Suppose that there exists a closed subgroup $H'$ of $G$ 
with finitely many connected components such that
\begin{itemize}
\item $H' \prec H$ in $G$.
\item $d(H') > d(H)$.
\end{itemize}
Then, $G/H$ does not admit a compact Clifford--Klein form.
\item Suppose that there exists a closed subgroup $H'$ of $G$ 
with finitely many connected components such that
\begin{itemize}
\item $H' \prec H$ in $G$.
\item $d(H') = d(H)$.
\item $G/H'$ does not admit a compact Clifford--Klein form.
\end{itemize}
Then, $G/H$ does not admit a compact Clifford--Klein form.
\end{enumerate}
\end{lem}

As we mentioned in Introduction, 
this lemma is proved by basically the same argument as 
Kobayashi's result~\cite[Th.~1.5]{Kob92}. 

\begin{proof}[Proof of Lemma~\ref{lem:Kob92-refine}]
We prove (2) only; the proof of (1) is essentially the same. 
Suppose that $\Gamma \bs G/H$ is a compact Clifford--Klein form. 
By Selberg's lemma~\cite[Lem.~8]{Sel60},
we may assume without loss of generality that $\Gamma$ is torsion-free.
By Fact~\ref{fact:cohomological-dimension}, we have 
$\cd_k \Gamma = d(G) - d(H)$.

On the other hand, by Fact~\ref{fact:proper-action-general}~(3) and (4),
$\Gamma$ acts properly on $G/H'$. 
Furthermore, this action is free since $\Gamma$ is torsion-free. 
By assumption, the Clifford--Klein form $\Gamma \bs G/H'$ is non-compact, 
hence Fact~\ref{fact:cohomological-dimension} implies that 
$\cd_k \Gamma < d(G) - d(H')$. 
This is impossible because $d(H) = d(H')$.
\end{proof}

\section{Proof of Theorem~\ref{thm:main-sl}}
\label{sect:proof-main-sl}

\begin{proof}[Proof of Theorem~\ref{thm:main-sl}]
We write $G/H = \SL(n, \K) / \SL(m, \K)$ and put 
$k = \lfloor m/2 \rfloor$. 

Define the closed subgroup $H'$ of $G$ as follows: 
\[
H' = 
\left\{ \begin{pmatrix} g & X \\ 0 & I_{n-k} \end{pmatrix}
\ \middle| \ g \in \SL(k, \K), X \in \M(k, n-k; \K) \right\}.
\]
By Proposition~\ref{prop:Kob92-2}~(1) and (2), 
it suffices to verify that the following three conditions are satisfied: 
\begin{enumerate}[label = (\alph*)]
\item $\mu(H') \subset \mu(H)$. 
\item $d(H') \geqslant d(H)$. 
\item $G/H'$ does not admit a compact Clifford--Klein form.
\end{enumerate}

Let us verify condition (b). 
For simplicity, we explain the real case only
(the proof for the complex and the quaternionic cases are similar). 
The maximal compact subgroups of $H$ and $H'$ are respectively 
$\SO(m)$ and $\SO(k)$, hence
\[
d(H) = \frac{m(m+1)}{2} - 1, \qquad 
d(H') = \frac{k(k+1)}{2} - 1 + k(n-k). 
\]
When $m \geqslant 2$ is even, we have $k = m/2$, and the inequality 
$d(H') \geqslant d(H)$ is rewritten as 
\[
\frac{m(m+2)}{8} - 1 + \frac{m}{2}\left( n- \frac{m}{2} \right) \geqslant \frac{m(m+1)}{2} - 1. 
\]
This is equivalent to
$n \geqslant 5m/4 + 1/2$. 
Since $n \in \N$, we can rephrase it as 
\[
n \geqslant 
\begin{cases}
5m/4 + 1 & \text{if $m \equiv 0 \pmod 4$}, \\
5m/4 + 1/2 & \text{if $m \equiv 2 \pmod 4$}. 
\end{cases}
\]
When $m \geqslant 3$ is odd, we have $k = (m-1)/2$, and the inequality 
$d(H') \geqslant d(H)$ is rewritten as 
\[
\frac{(m-1)(m+1)}{8} - 1 + \frac{m-1}{2}\left( n- \frac{m-1}{2} \right) \geqslant \frac{m(m+1)}{2} - 1. 
\]
This is equivalent to 
\[
n \geqslant \frac{5m}{4} + \frac{5}{4} + \frac{2}{m-1}. 
\]
Since $n \in \N$, we can rephrase it as 
\[
n \geqslant 
\begin{cases}
5m/4 + 7/4 & \text{if $m \equiv 0 \pmod 4$}, \\
5m/4 + 9/4 & \text{if $m \equiv 2 \pmod 4$}. 
\end{cases}
\]
Let us see that condition (c) follows from Proposition~\ref{prop:trace-free}. 
Since $G$ is unimodular (i.e.\ admits a bi-invariant volume form), 
we have 
\[
(\Lambda^N \g)^\g \cong \Omega^N (G)^{G \times G} \neq 0 \qquad (N = \dim G), 
\]
hence the $\g$-action on $\g$ is trace-free. 
Similarly, since $H'$ is unimodular, the $\h'$-action on $\h'$ is trace-free. 
Thus, the $\h'$-action on $\g/\h'$ is trace-free. Meanwhile, 
\[
X_0 = \begin{pmatrix} (n-k) I_k & 0 \\ 0 & -k I_{n-k} \end{pmatrix}
\]
is an element of $\mf{n}_\g(\h')$ 
whose adjoint action on $\g/\h'$ is given by a non-trivial homothety: 
\[
\ad_{\g/\h'}(X_0)(X + \h') = -n (X + \h'). 
\]
Hence, the action of $\mf{n}_\g(\h')$ on $\g/\h'$ is not trace-free. 

To prove condition (a), consider 
the following closed subset $S$ of $H'$: 
\begin{align*}
S &= 
\left\{
\begin{pmatrix}
g & \diag(t_1, \dots, t_k) & 0 \\
0 & I_k & 0 \\
0 & 0 & I_{n-2k}
\end{pmatrix}
\ \middle| \ g \in \SL(k, \K), t_1, \dots, t_k \in \R 
\right\}.
\end{align*}
Note that $S \subset H$, hence $\mu(S) \subset \mu(H)$. 
It suffices to verify $\mu(S) = \mu(H')$. 
Take any 
\[
h' = \begin{pmatrix} g & X \\ 0 & I_{n-k} \end{pmatrix} \in H'. 
\]
By the singular value decomposition for $\M(k,n-k; \K)$, we see that: 
\begin{itemize}
\item When $\K = \R$, there exist
$\ell \in \SO(k)$, $\ell' \in \SO(n-k)$, and $t_1, \dots, t_k \in \R$ 
such that
\[
\ell X \ell'^{-1} = 
\begin{pmatrix} \diag(t_1, \dots, t_k) & 0 \end{pmatrix}. 
\]
\item When $\K = \C$, there exist
$\ell \in \U(k)$, $\ell' \in \U(n-k)$, and $t_1, \dots, t_k \in \R$ 
such that $\det (\ell) \det (\ell') = 1$ and 
\[
\ell X \ell'^{-1} = 
\begin{pmatrix} \diag(t_1, \dots, t_k) & 0 \end{pmatrix}. 
\]
\item When $\K = \Ha$, there exist
$\ell \in \Sp(k)$, $\ell' \in \Sp(n-k)$, and $t_1, \dots, t_k \in \R$ 
such that
\[
\ell X \ell'^{-1} = 
\begin{pmatrix} \diag(t_1, \dots, t_k) & 0 \end{pmatrix}. 
\]
\end{itemize}
In any case, 
\[
\begin{pmatrix} \ell & 0 \\ 0 & \ell' \end{pmatrix}
h'
\begin{pmatrix} \ell & 0 \\ 0 & \ell' \end{pmatrix}^{-1}
= 
\begin{pmatrix}
\ell g \ell^{-1} & \diag(t_1, \dots, t_k) & 0 \\
0 & I_k & 0 \\
0 & 0 & I_{n-2k}
\end{pmatrix}
\in S. 
\]
This shows $\mu(S) = \mu(H')$. 
\end{proof}

\begin{rmk}
When $m$ is odd, we can alternatively prove Theorem~\ref{thm:main-sl} 
in the following way (we explain the real case only). 
Let $k = \lfloor m/2 \rfloor$, and define the closed subgroup $H''$ of $G$ 
as follows: 
\[
H'' = 
\left\{
\begin{pmatrix}
g & v & X \\ 0 & \det g^{-1} & 0 \\ 0 & 0 & I_{n-k-1}
\end{pmatrix}
\ \middle| \ 
\begin{matrix} 
g \in \GL(k, \R), v \in \R^k, \\ X \in \M(k, n-k-1; \R)
\end{matrix}
\right\}.
\]
We can see as in the previous proof that: 
\begin{enumerate}[label = (\alph*)]
\item $\mu(H'') \subset \mu(H)$. 
\item $d(H'') > d(H)$. 
\end{enumerate}
By Proposition~\ref{prop:Kob92-2}~(1), 
$G/H$ does not admit a compact Clifford--Klein form. 
The author does not know if $G/H''$ 
admits a compact Clifford--Klein form or not. 
\end{rmk}

\begin{rmk}[cf.\ Remark~\ref{rmk:sl-other}~(i)]\label{rmk:compare-Kob92}
In \cite{Kob92}, \cite{Kob97}, Kobayashi used 
Fact~\ref{fact:Kob92-1}. He compared $\SL(m, \K)$ with 
\[
H'_\text{Kob} = 
\begin{cases}
\SO(k, n-k) & \text{if $\K = \R$}, \\
\SU(k, n-k) & \text{if $\K = \C$}, \\
\Sp(k, n-k) & \text{if $\K = \Ha$},
\end{cases}
\]
where $k = \lfloor m/2 \rfloor$. 
We have improved his result by comparing $\SL(m, \K)$ 
with the non-reductive subgroup $H'$, 
rather than the reductive subgroup $H'_\text{Kob}$. 
\end{rmk}

\begin{rmk}[cf.\ Remark~\ref{rmk:sl-other}~(viii)]\label{rmk:compare-Yos07RepSymp}
In \cite{Yos07RepSymp}, Yoshino used an analogue of 
Fact~\ref{fact:Kob92-1} for Cartan motion groups 
(\cite[Th.~9]{Yos07IJM}). 
For instance, when $\K = \R$ and $m$ is even, 
he compared $\SL(m, \R)_\theta$ with 
\[
H'_\text{Yos} = 
\{ 1 \} \ltimes \left\{ 
\begin{pmatrix} A & X \\ {^t}\overline{X} & 0 \end{pmatrix}
\ \middle| \ 
\begin{matrix}
A \in \Sym(k, \R), \trace A = 0, \\ X \in \M(k, n-k; \R)
\end{matrix}
\right\}, 
\]
where $k = \lfloor m/2 \rfloor$. 
Our choice of $H'$ is inspired by $H'_\text{Yos}$. 
\end{rmk}

\section{Proof of Theorem~\ref{thm:main-associated}}
\label{sect:proof-main-associated}

In this section, we give a proof of Theorem~\ref{thm:main-associated}. 
To exhibit our main ideas concretely, 
we first give a proof of one particular example of this theorem: 

\begin{prop}[Example~\ref{ex:sl-so}]\label{prop:sl-so}
If $p,q \geqslant 1$, then 
$\SL(p+q, \R)/\SO_0(p,q)$ does not admit a compact Clifford--Klein form. 
\end{prop}

\begin{proof}[Proof of Proposition~\ref{prop:sl-so}]
We write $G/H = \SL(p+q, \R) / \SO_0(p,q)$. 
Define the closed subgroup $H'$ of $G$ by 
\[
H' = 
\left\{ \begin{pmatrix} k & X \\ 0 & \ell \end{pmatrix}
\ \middle| \ X \in \M(p,q; \R), k \in \SO(p), \ell \in \SO(q) \right\}. 
\]
By Proposition~\ref{prop:Kob92-2}~(2), 
it suffices to verify that the following three conditions are satisfied: 
\begin{enumerate}[label = (\alph*)]
\item $\mu(H') = \mu(H)$. 
\item $d(H') = d(H)$. 
\item $G/H'$ does not admit a compact Clifford--Klein form.
\end{enumerate}

The maximal compact subgroups of $H$ and $H'$ are both 
$\SO(p) \times \SO(q)$. Hence $d(H') = d(H) = pq$, 
and condition (b) holds. 

Condition (c) follows from Proposition~\ref{prop:trace-free}, 
as in the proof of Theorem~\ref{thm:main-sl}. 
Indeed, the $\h'$-action on $\g/\h'$ is trace-free since 
$G$ and $H'$ are unimodular, while
\[
X_0 = \begin{pmatrix} q I_p & 0 \\ 0 & -p I_q \end{pmatrix}
\]
is an element of $\mf{n}_\g(\h')$ such that 
$\trace (\ad_{\g/\h'} (X_0)) = -pq(p+q) \ (\neq 0)$. 

Condition (a) was proved by Kobayashi~\cite[Ex.~6.5]{Kob96}. 
In order to explain ideas, we reproduce his proof here. 
Assume without loss of generality that $p \leqslant q$. 
Define the closed embedding of Lie groups 
\[
i \colon \underbrace{\SL(2, \R) \times \dots \times \SL(2, \R)}_p 
\to \SL(p+q, \R)
\]
by
\begin{align*}
&\ \ \ i\left(
\begin{pmatrix} a_1 & b_1 \\ c_1 & d_1 \end{pmatrix},
\dots,
\begin{pmatrix} a_p & b_p \\ c_p & d_p \end{pmatrix} 
\right) \\
&=
\begin{pmatrix}
\diag(a_1, \dots, a_p) & \diag(b_1, \dots, b_p) & 0 \\
\diag(c_1, \dots, c_p) & \diag(d_1, \dots, d_p) & 0 \\
0 & 0 & I_{q-p}
\end{pmatrix},
\end{align*}
and put $G' = i(\SL(2, \R) \times \dots \times \SL(2, \R))$.
Define the closed subgroup $A'$ of $G'$ by 
\[
A' = i(A_{\SL(2, \R)} \times \dots \times A_{\SL(2, \R)}), \quad
A_{\SL(2, \R)} = 
\left\{
\begin{pmatrix}
\cosh t & \sinh t \\ \sinh t & \cosh t 
\end{pmatrix}
\ \middle| \ 
t \in \R
\right\}. 
\]
Since $\mf{a}'$ is a maximal abelian subspace of $\p \cap \g'$, 
we have $\mu(G') = \mu(A')$ by the $KAK$ decomposition of $G'$.
Notice that $\mf{a}'$ is also a maximal abelian subspace of 
$\p \cap \h$, and $\mu(H) = \mu(A')$. Let 
\[
U' = i(U_{\SL(2, \R)} \times \dots \times U_{\SL(2, \R)}), 
\quad
U_{\SL(2, \R)} = \left\{ \begin{pmatrix} 1 & t \\ 0 & 1 \end{pmatrix} 
\ \middle| \ t \in \R \right\}.
\]
Since
\[
\SO(2) \cdot U_{\SL(2, \R)} \cdot \SO(2) = \SL(2, \R), 
\]
we have $\mu(G') = \mu(U')$. 
On the other hand, it follows from the singular value decomposition for 
$\M(p,q; \R)$ that $\mu(H') = \mu(U')$. 
We have thus proved $\mu(H') = \mu(H)$. 
\end{proof}

\begin{rmk}\label{rmk:use-U-1}
One may use 
\[
U = \left\{ \begin{pmatrix} I_p & X \\ 0 & I_q \end{pmatrix} 
\ \middle| \ X \in \M(p,q; \R) \right\}
\]
instead of $H'$ in the proof of Example~\ref{ex:sl-so}. 
\end{rmk}

Now, we prove Theorem~\ref{thm:main-associated} in full generality, 
generalizing the above proof: 

\begin{proof}[Proof of Theorem~\ref{thm:main-associated}]
We use the notation of Section~\ref{sect:preliminaries-red}. 
We can assume without loss of generality that $\h^a = \mf{l}_{\Pi'}$ 
for some non-empty subset $\Pi'$ of the simple system $\Pi$. 
Let $\mf{u} = \mf{u}_{\Pi'}$ (resp.\ $U = U_{\Pi'}$) 
be the corresponding horospherical subalgebra of $\g$ 
(resp.\ horospherical subgroup of $G$). 
Let us define a subalgebra $\h'$ of $\g$ by 
$\h' = (\mf{k} \cap \h) \ltimes \mf{u}$. Here, $\mf{k} \cap \h$ 
normalizes $\mf{u}$ since $\mf{k} \cap \h \subset \h^a$. 
Let $H'$ be the connected Lie subgroup of $G$ corresponding to $\h'$. 
We see that $H' = (K \cap H)_0 \ltimes U$, where $(K \cap H)_0$ 
is the identity component of $K \cap H$. 
In particular, $H'$ is closed in $G$. 
By Proposition~\ref{prop:Kob92-2}~(2), 
it suffices to verify that the following three conditions are satisfied: 
\begin{enumerate}[label = (\alph*)]
\item $\mu(H') = \mu(H)$. 
\item $d(H') = d(H)$. 
\item $G/H'$ does not admit a compact Clifford--Klein form.
\end{enumerate}

By Lemma~\ref{lem:symmetric-levi}~(1) (i)~$\Rightarrow$~(ii), 
\[
f \colon \mf{u} \to \p \cap \h, \qquad X \mapsto X - \theta X
\]
is a linear isomorphism. It follows that
\[
d(H') = \dim \mf{u} = \dim (\p \cap \h) = d(H), 
\]
i.e.\ condition (b) holds. 

Condition (c) follows from Proposition~\ref{prop:trace-free}. 
Indeed, the $\h'$-action on $\g/\h'$ is trace-free since $G$ and $H'$ 
are unimodular, while any non-zero element $X$ of 
$\mf{z}(\h^a) \cap [\g, \g] \ (\subset \mf{n}_\g(\h'))$ 
satisfies $\trace (\ad_{\g/\h'} (X)) \neq 0$ 
(\cite[Ex.~7.1]{Mor17PRIMS}). 

Let us verify condition (a). Fix a total order $\leqslant$ on 
$\Sigma$ compatible with the simple system $\Pi$. 
We choose $\lambda_1, \dots, \lambda_r \in \Sigma_{\Pi', +}$ and 
$X_1 \in \g_{\lambda_1}, \dots, X_r \in \g_{\lambda_r}$ as follows
(cf.\ \cite[Ch.~VIII, Proof of Prop.~7.4]{Hel62}): 
put $\g^{(0)}_\lambda = \g_\lambda$ for each 
$\lambda \in \Sigma_{\Pi', +}$, 
and continue the following procedure inductively on $k$ until 
$\Sigma^{(k)}_{\Pi', +}$ becomes the empty set: 
\begin{itemize}
\item Put 
$\Sigma^{(k)}_{\Pi', +} = 
\{ \lambda \in \Sigma_{\Pi', +} \mid \g^{(k)}_\lambda \neq 0 \}$. 
Let $\lambda_{k+1}$ be the lowest root in 
$\Sigma^{(k)}_{\Pi', +}$ with respect to $\leqslant$. 
Choose any non-zero element 
$X_{k+1}$ of $\g^{(k)}_{\lambda_{k+1}}$.
For each $\lambda \in \Sigma_{\Pi', +}$, put
$\g^{(k+1)}_\lambda = 
\{ X \in \g^{(k)}_\lambda \mid [X_{k+1}, \theta X] = 0 \}.$
\end{itemize}
Lemma~\ref{lem:in-p}~(1) ensures that this procedure terminates after 
a finite number of steps, and $\lambda_1 < \dots < \lambda_r$. 

Let us define a subspace $\g'$ of $\g$ by 
\[
\g' = \g'_1 \oplus \dots \oplus \g'_r, \qquad 
\g'_k = \R X_k \oplus \R \theta X_k \oplus \R [X_k, \theta X_k] \quad 
(1 \leqslant k \leqslant r). 
\]
For each $1 \leqslant k \leqslant r$, 
$\g'_k$ is a subalgebra of $\g$ isomorphic to $\mf{sl}(2, \R)$ 
(see e.g.\ Knapp~\cite[Prop.~6.52]{Kna02}). 
Moreover, the following holds: 

\begin{lem}\label{lem:sl2}
$\g'$ is a subalgebra of $\g$, and 
$\g' = \g'_1 \oplus \dots \oplus \g'_r$ 
is a direct sum decomposition of Lie algebras. 
\end{lem}
\begin{proof}[Proof of Lemma~\ref{lem:sl2}]
We need to verify $[\g'_k, \g'_{\ell}] = 0$ for $k \neq \ell$. 
It is enough to see that:
\begin{align*}
&[X_k, X_\ell] = 0, \qquad 
[X_k, \theta X_\ell] = 0, \qquad
[X_k, [X_\ell, \theta X_\ell]] = 0, \\
&[[X_k, \theta X_k], [X_\ell, \theta X_\ell]] = 0.
\end{align*}
The first equality follows from 
Lemma~\ref{lem:symmetric-levi}~(1) (i)~$\Rightarrow$~(iii). 
The second equality is immediate from the definition of 
$(X_1, \dots, X_r)$.
The remaining two equalities are obtained by the Jacobi identity. 
\end{proof}

Let $G'$ be the connected Lie subgroup of $G$ corresponding to $\g'$. 
Since $\g'$ is semisimple and $G$ is linear, $G'$ is closed in $G$ 
(Yosida~\cite{Yos38}). Note that $G'$ is $\theta$-stable. 

Let us define a subalgebra $\mf{a}'$ of $\g'$ by 
\[
\mf{a}' = 
\R (X_1 - \theta X_1) \oplus \dots \oplus \R (X_r - \theta X_r). 
\]

\begin{lem}\label{lem:maximal-abelian}
\begin{enumerate}[label = (\arabic*)]
\item $\mf{a}'$ is a maximal abelian subspace of $\p \cap \g'$. 
\item $\mf{a}'$ is a maximal abelian subspace of $\p \cap \h$.
\end{enumerate}
\end{lem}
\begin{proof}[Proof of Lemma~\ref{lem:maximal-abelian}]
(1) This is immediate from Lemma~\ref{lem:sl2}. 

(2) Let $Y \in \p \cap \h$. By Lemma~\ref{lem:symmetric-levi}~(1) 
(i)~$\Rightarrow$~(ii), we can express $Y$ in the form 
\[
Y = \sum_{\lambda \in \Sigma_{\Pi', +}} (Y_\lambda - \theta Y_\lambda) 
\qquad (Y_\lambda \in \g_\lambda).
\]
Recall that $\mf{u}$ is abelian 
(Lemma~\ref{lem:symmetric-levi}~(1) (i)~$\Rightarrow$~(iii)). 
The condition $[Y, \mf{a}'] = 0$ can be rewritten as follows: 
\[
\sum_{\lambda \in \Sigma_{\Pi', +}} 
([Y_\lambda, \theta X_k] + [\theta Y_\lambda, X_k]) = 0 
\qquad (1 \leqslant k \leqslant r). \tag{$\ast$}\label{tag:Yj}
\]
It suffices to show that (\ref{tag:Yj}) holds if and only if 
the following two conditions are satisfied: 
\begin{enumerate}[label = (\roman*)]
\item $Y_{\lambda_k} \in \R X_k$ for each $1 \leqslant k \leqslant r$.
\item $Y_\lambda = 0$ for each $\lambda \in \Sigma_{\Pi', +} 
\smallsetminus \{ \lambda_1, \dots, \lambda_r \}$.
\end{enumerate}
The `if' part is obvious from the definition of $(X_1, \dots, X_r)$. 
Let us prove the `only if' part. 
By looking at the $\g_0$-component of (\ref{tag:Yj}), we obtain 
\[
[Y_{\lambda_k}, \theta X_k] + [\theta Y_{\lambda_k}, X_k] = 0 \qquad 
(1 \leqslant k \leqslant r).
\]
In other words, $[Y_{\lambda_k}, \theta X_k] \in \p$ for 
$1 \leqslant k \leqslant r$.
By Lemma~\ref{lem:in-p}~(2), condition (i) holds. 
Suppose that there were $Y \in \p \cap \h$ that satisfies (\ref{tag:Yj}) 
but does not satisfy condition (ii). 
Let $\nu$ be the lowest root in $\Sigma_{\Pi', +} \smallsetminus 
\{ \lambda_1, \dots, \lambda_r \}$ such that 
$Y_\nu \neq 0$. 
Let $s$ be a largest number such that 
$1 \leqslant s \leqslant r$ and $\lambda_{s} < \nu$.
For any $1 \leqslant k \leqslant s$ and any 
$\lambda \in \Sigma_{\Pi', +} \smallsetminus 
\{ \lambda_1, \dots, \lambda_r \}$ with $Y_\lambda \neq 0$, 
the inequality $-\lambda + \lambda_k < \nu - \lambda_k$ holds. 
Thus, by looking at the $\g_{\nu - \lambda_k}$-component of 
(\ref{tag:Yj}), we obtain
\[
[Y_\nu, \theta X_k] = 0 \qquad (1 \leqslant k \leqslant s), 
\]
i.e.\ $\nu \in \Sigma^{(s)}_{\Pi', +}$. 
If $s = r$, this contradicts to $\Sigma^{(r)}_{\Pi', +} = \varnothing$. 
If $s < r$, then $\lambda_{s+1} \leqslant \nu$ by the definition of 
$\lambda_{s+1}$, and this contradicts to the definition of $s$.
\end{proof}

Let $A'$ be the connected Lie subgroup of $G'$ corresponding to 
$\mf{a}'$. 
It is a closed subgroup of $G'$ isomorphic to $\R^r$. 

\begin{lem}\label{lem:mu-a'}
\begin{enumerate}[label = (\arabic*)]
\item $\mu(G') = \mu(A')$. 
\item $\mu(H) = \mu(A')$. 
\end{enumerate}
\end{lem}
\begin{proof}[Proof of Lemma~\ref{lem:mu-a'}]
(1) This follows from Lemma~\ref{lem:maximal-abelian}~(1) 
and the $KAK$ decomposition of $G'$.

(2) This follows from Lemma~\ref{lem:maximal-abelian}~(2) 
and the $KAK$ decomposition of $H$.
\end{proof}

Define a subalgebra $\mf{u}'$ of $\g'$ by 
\[
\mf{u}' = \R X_1 \oplus \dots \oplus \R X_r 
\]
and let $U'$ be the connected Lie subgroup of $G'$ corresponding to 
$\mf{u}'$. 
It is a closed subgroup of $G'$ isomorphic to $\R^r$. 

\begin{lem}\label{lem:mu-u'}
\begin{enumerate}[label = (\arabic*)]
\item $\mu(G') = \mu(U')$. 
\item $\mu(H') = \mu(U')$. 
\end{enumerate}
\end{lem}
\begin{proof}[Proof of Lemma~\ref{lem:mu-u'}]
(1) Since $G'$ is locally isomorphic to the product of $r$ copies of 
$\SL(2, \R)$, this follows from 
\[
\SO(2) \cdot U_{\SL(2, \R)} \cdot \SO(2) = \SL(2, \R),
\]
where 
\[
U_{\SL(2, \R)} = 
\left\{ \begin{pmatrix} 1 & t \\ 0 & 1 \end{pmatrix} \ \middle| \ 
t \in \R \right\}.
\]

(2) The linear isomorphism 
\[
f \colon \mf{u} \to \p \cap \h, \qquad X \mapsto X - \theta X
\]
is $(\mf{k} \cap \h)$-equivariant and thus $(K \cap H)_0$-equivariant. 
Since $f(\mf{u}') = \mf{a}'$ and 
$\p \cap \h = \Ad((K \cap H)_0) \mf{a}'$, we have 
$\mf{u} = \Ad((K \cap H)_0) \mf{u}'$. This can be rephrased as 
\[
U = \bigcup_{k \in (K \cap H)_0} kU'k^{-1}. 
\]
Thus, $\mu(H') = \mu(U) = \mu(U')$. 
\end{proof}

We can conclude from Lemmas~\ref{lem:mu-a'} and \ref{lem:mu-u'} that 
$\mu(H') = \mu(H)$, i.e.\ condition (a) holds. 
Theorem~\ref{thm:main-associated} is now proved. 
\end{proof}

\begin{rmk}[cf.\ Remark~\ref{rmk:use-U-1}]
One may use $U$ instead of $H'$ in the proof of 
Theorem~\ref{thm:main-associated}. 
\end{rmk}

\section{Some remarks and open questions}
\label{sect:remarks-questions}

\subsection{Enlargement of Lie groups}

We can refine Theorems~\ref{thm:main-sl} and \ref{thm:main-associated} 
as follows: 

\begin{prop}\label{prop:enlarge}
Let $G/H$ be either as in Theorem~\ref{thm:main-sl} or as in 
Theorem~\ref{thm:main-associated}. 
Let $\widetilde{G}$ be a linear unimodular Lie group containing 
$G$ as a closed subgroup. 
Let $L$ be a unimodular closed subgroup of $\widetilde{G}$ such that 
$G \cap L = \{ 1 \}$ and $L \subset Z_{\widetilde{G}}(G)$. 
Assume that $\widetilde{G}$ and $L$ 
have finitely many connected components. 
Then, $\widetilde{G}/(H \times L)$
does not admit a compact Clifford--Klein form.
\end{prop}
\begin{proof}
We explain the case where $G/H$ is as in 
Theorem~\ref{thm:main-associated}. 
The other case is proved in the same way. 
In Section~\ref{sect:proof-main-associated}, 
we defined $H'$ and saw that: 
\begin{enumerate}[label = (\alph*)]
\item $\mu(H') = \mu(H)$. 
\item $d(H') = d(H)$. 
\item $G/H'$ does not admit a compact Clifford--Klein form.
\end{enumerate}
By Lemma~\ref{lem:Kob92-refine}~(2), to prove the proposition, 
it suffices to see that: 
\begin{enumerate}[label=(\alph*$'$)]
\item $H \times L \sim H' \times L$ in $\widetilde{G}$. 
\item $d(H \times L) = d(H' \times L)$. 
\item $\widetilde{G} / (H' \times L)$ 
does not admit a compact Clifford--Klein form. 
\end{enumerate}

By (a) and Fact~\ref{fact:proper-action-general}~(vi), 
We have $H' \sim H$ in $G$. 
It follows from Fact~\ref{fact:proper-action-general}~(v) 
that condition (a$'$) holds. 

Condition (b$'$) follows immediately from (b). Indeed, 
\[
d(H \times L) = d(H) + d(L), \qquad d(H' \times L) = d(H') + d(L). 
\]

Finally, condition (c$'$) is proved in the same way as (c): 
the $(\h' \oplus \mf{l})$-action on $\widetilde{\g}/(\h' \oplus \mf{l})$ 
is trace-free since $\widetilde{G}$, $H'$, and $L$ are unimodular, 
while any non-zero element $X$ of 
$\mf{z}(\h^a) \cap [\g, \g] \ 
(\subset \mf{n}_{\widetilde{\g}}(\h' \oplus \mf{l}))$ 
satisfies 
$\trace (\ad_{\widetilde{\g}/(\h' \oplus \mf{l})} (X)) \neq 0$, 
hence we can apply Proposition~\ref{prop:trace-free}. 
\end{proof}

\begin{rmk}
In Proposition~\ref{prop:enlarge}, 
we assumed the unimodularity of $\widetilde{G}$ and $L$
so that we can prove the non-existence of 
compact Clifford--Klein forms of $\widetilde{G} / (H' \times L)$.  
When $d(H') > d(H)$, we do not need this unimodularity assumption
since we can use 
Lemma~\ref{lem:Kob92-refine}~(1) instead of 
Lemma~\ref{lem:Kob92-refine}~(2). 
\end{rmk}

\subsection{An interpretation in terms of Cartan motion groups}

To apply Proposition~\ref{prop:Kob92-2}, 
we need to find a closed subgroup with small Cartan projection. 
Usually, the Cartan projections of non-reductive subgroups are 
difficult to compute. One exception is the case of Cartan motion groups.

Let $L$ be a reductive Lie group and 
$L_\theta = K_L \ltimes \p_L$ the corresponding Cartan motion group. 
Suppose that $L_\theta$ is embedded into a reductive Lie group $G$ 
as a closed subgroup. 
Without loss of generality, we choose the Cartan involution 
$\theta$ of $G$ so that $\theta|_{K_L} = 1$.
Fix $\mf{a}$ and $\Pi$ as in Section~\ref{sect:preliminaries-red}, 
and let $\mu \colon G \to \mf{a}^+$ be the Cartan projection of $G$. 
Take a maximal abelian subspace $\mf{a}_L$ of $\p_L$, 
which we regard as a closed subgroup of $G$. 
We see that $\Ad(K_L) \mf{a}_L = \p_L$, and therefore, 
$\mu(L_\theta) = \mu(\mf{a}_L)$. 
We expect that $\mu(\mf{a}_L)$ can be computed explicitly 
(or estimated from the above) in some cases. 

The closed subgroup $H'$ in Section~\ref{sect:proof-main-associated} 
is an example of such $L_\theta$.
Also, while the closed subgroups $H'$ and $H''$ in 
Section~\ref{sect:proof-main-sl} are not exactly Cartan motion groups, 
the idea is in the same spirit as the above discussion. 
This raises the following question: 

\begin{ques}
Are there any other examples of embeddings of Cartan motion groups into 
reductive Lie groups (or embeddings similar to them) 
that can be used to obtain non-existence results of 
compact Clifford--Klein forms via Proposition~\ref{prop:Kob92-2}?
\end{ques}

\subsection{An interpretation in terms of conjugacy limits}

The set $\mathcal{C}(G)$ of all closed subgroups in a Lie group $G$ 
admits a natural compact and metrizable topology, 
called the \emph{Chabauty topology}.
A sequence of closed subgroups $(H_n)_{n \in \N}$ of $G$
converges to a closed subgroup $H'$ if and only if 
the following two conditions hold 
(see e.g.\ \cite[Prop.~1.8~(1)]{Pau07}): 
\begin{itemize}
\item For any $h' \in H'$, there exists a sequence $(h_n)_{n \in \N}$ 
such that $h_n \in H_n$ for each $n \in \N$ and 
$h_n \to h'$ as $n \to +\infty$. 
\item For any sequence $(h_n)_{n \in \N}$ 
such that $h_n \in H_n$ for each $n \in \N$, 
every accumulation point of $(h_n)_{n \in \N}$ is in $H'$. 
\end{itemize}

As in the proof of Proposition~\ref{prop:sl-so}, 
put $G = \SL(p+q, \R), H = \SO_0(p,q)$, 
\[
H' = \left\{ \begin{pmatrix}k & X \\ 0 & \ell \end{pmatrix} \ \middle| \ 
X \in \M(p,q; \R), k \in \SO(p), \ell \in \SO(q) \right\}, 
\]
and 
\[
X_0 = \begin{pmatrix} q I_p & 0 \\ 0 & -pI_q \end{pmatrix}. 
\]
Notice that 
\[
\exp(tX_0) H \exp(-tX_0) \xrightarrow{t \to +\infty} H' 
\tag{$\dagger$}\label{tag:hyperbolic}
\]
with respect to the Chabauty topology. In the language of 
Cooper--Danciger--Wienhard~\cite[Def.~2.5~(2)]{CDW18}, 
$H'$ is a \emph{conjugacy limit} of $H$. Similarly,
in a more general setting that we considered in 
the proof of Theorem~\ref{thm:main-associated}, 
$H'$ is a conjugacy limit of $H$ in $G$. 
Note that, in general, the Cartan projection and 
the non-compact dimension of a conjugacy limit can be much larger than 
those of the original closed subgroup. For instance, 
in the setting of the proof of Proposition~\ref{prop:sl-so}, 
\[
\exp(tX_0) \SO(p+q) \exp(-tX_0) \xrightarrow{t \to +\infty} H' 
\tag{$\ddagger$}\label{tag:elliptic}
\]
holds, although $\mu(H') \ (= \mu(H))$ and $d(H') \ (= pq)$ 
are much larger than $\mu(\SO(p+q)) \ (= \{ 0 \})$ and 
$d(\SO(p+q)) \ (= 0)$, repsectively. 

Intuitively, the difference between (\ref{tag:hyperbolic}) and 
(\ref{tag:elliptic}) can be explained as follows. 
In the case of (\ref{tag:hyperbolic}), as $t$ tends to infinity, 
the split subspace $\exp(\p \cap \h)$ of $H$ 
collapses to the unipotent subgroup
\[
U = \left\{ \begin{pmatrix} I_p & X \\ 0 & I_q \end{pmatrix}
\ \middle| \ X \in \M(p,q; \R) \right\} 
\]
of $H'$, while the maximal compact subgroup 
$K \cap H = \SO(p) \times \SO(q)$ of $H$ is unchanged. 
Since the collapse from `split' to `unipotent' 
does not change the `non-compactness', 
the Cartan projections and the non-compact dimensions are 
expected to be unchanged. 
In contrast, in the case of (\ref{tag:elliptic}), 
the closed subgroup $\SO(p+q)$ is compact. 
As $t$ tends to infinity, a part of this compact subgroup 
collapses to the unipotent subgroup $U$ of $H'$. 
Since the collapse from `compact' to `unipotent' 
increases the `non-compactness', 
the Cartan projection and the non-compact dimension are 
expected to become larger. 
Now, this raises the following question: 

\begin{ques}
\begin{enumerate}[label = (\arabic*)]
\item Are there any other examples of a pair $(G/H, X_0)$ 
that satisfies the following two conditions?
\begin{itemize}
\item $G/H$ is a homogeneous space of reductive type. 
\item $X_0 \in \mf{z}_\g(\mf{k}_H)$, where $\mf{k}_H$ 
is the Lie algebra of the maximal compact subgroup of $H$. 
\item $X_0 \notin \mf{z}_\g(\h)$. 
\end{itemize}
\item For a pair $(G/H, X_0)$ as in (1), consider the conjugacy limit 
\[
H' = \lim_{t \to +\infty} \exp(tX_0) H \exp(-tX_0). 
\]
Do $\mu(H) = \mu(H')$ and $d(H) = d(H')$ hold? 
Can we prove the non-existence of compact Clifford--Klein forms of $G/H$ 
using Proposition~\ref{prop:Kob92-2}~(2)?
\end{enumerate}
\end{ques}

\subsection*{Acknowledgements}
The initial idea of this paper was developed 
when the author was a graduate student of Toshiyuki Kobayashi. 
I would like to thank him for his advice and encouragement. 
I would like to thank Fanny Kassel, Masatoshi Kitagawa, 
and Taro Yoshino for valuable discussions. 
This work was supported by
JSPS KAKENHI Grant Numbers 14J08233, 17H06784, 19K14529,
the Kyoto University Research Fund for Young Scientists (Start-up),
and the Program for Leading Graduate Schools, MEXT, Japan.

\end{document}